\documentclass[a4paper,11pt]{amsart} 
\usepackage[english]{babel}
\usepackage{dsfont} % \mathds
\usepackage{graphicx}
\usepackage{color}
\usepackage{enumerate}
\usepackage{array,multirow}

\allowdisplaybreaks

\usepackage{amssymb} % \upharpoonright
\usepackage{amsthm} % \newtheorem
\usepackage{dsfont}

\newtheorem{thm}{Theorem}
\newtheorem{dfn}[thm]{Definition}
\newtheorem{lem}[thm]{Lemma}

\newtheorem{exm}[thm]{Example}

\newtheorem{prop}[thm]{Proposition}
\newtheorem{rem}[thm]{Remark}
\newtheorem{cor}[thm]{Corollary}

\newcommand{\dN}{\mathds{N}}

\newcommand{\dR}{\mathds{R}}
\newcommand{\dC}{\mathds{C}}

\newcommand{\cJ}{\mathcal{J}\,}

\newcommand{\cM}{\mathcal{M}\,}

\newcommand{\cD}{\mathcal{D}\,}

\newcommand{\Pos}{\mathrm{Pos}}

\newcommand{\ov}{\overline}

\newcommand{\supp}{\mathrm{supp}\,}

\newcommand{\cA}{\mathcal{A}}

\newcommand{\ii}{\rm i}

\newcommand{\cN}{\mathcal{N}}

\newcommand{\cZ}{\mathcal{Z}}

\newcommand{\sA}{\mathsf{A}}

\newcommand{\im}{\mathsf{i}}

\begin{document}

\begin{title}{On Moment Functionals with Signed Representing Measures}
\end{title}
\author{Konrad Schm\"udgen}
\address{University of Leipzig, Mathematical Institute, Augustusplatz 10/11, D-04109 Leipzig, Germany}
\email{\tt schmuedgen@math.uni-leipzig.de}
%\begin{document}

\date{\today}

\dedicatory{\today}

\begin{abstract}
Suppose that $\sA$ is a finitely generated commutative unital real algebra and $K$ is a closed subset of the set $\hat{A}$ of characters of $\sA$. We study the following problem: When is {\it each} linear functional $L:{\sA} \to \dR$   an integral with respect to some signed Radon measure on $\hat{\sA}$ supported by the set $K$? A complete characterization of the sets $K$ and algebras $\sA$  by necessary and sufficient conditions is given. The result is applied to the polynomial algebra $\dR[x_1,\dots,x_d]$ and subsets $K$ of $\dR^d$.
\end{abstract}


\maketitle
%\date{}
%\begin{document}

\maketitle
\section{Introduction}

R. P. Boas \cite{boas} showed that {\it each} real sequence is the moment sequence of some {\it signed} Radon measure on $\dR$, see also \cite[Theorem 3.11]{shohat} for a proof. In terms of functionals this means that each linear functional on the polynomial algebra $\dR[x]$ can be represented as an integral by some signed Radon measure on $\dR$. This result was sharpened  by G. Polya  \cite{polya} who proved that  the  measure can be chosen such that its support is  any  real sequence without finite cluster points. A. Duran \cite{duran} studied the case when the signed measure has a Schwartz space density. T. Sherman \cite{sherman} generalized Boas' theorem to $d$-sequences on $\dR^d$ and on $[0,+\infty)^d$ for $d\in \dN$. 

In this paper we investigate the following problem: \smallskip

 {\it  Which closed subsets $K$ of $\dR^d$ have the property that each real $d$-sequence is a  moment sequence with a signed representing measure whose support is contained in $K$}?
\smallskip

As it is common,  moment problems are studied in terms of   functionals rather than moment sequences. To each $d$-sequence $s=(s_n)_{n\in \dN_0^d}$ one associates  a linear functional 
$L_s$, called the Riesz functional of $s$, on the polynomial algebra $\dR[x_1,\dots,x_d]$ by $L_s(x^n)=s_n, n\in \dN_0^d$. Then moment $d$-sequences correspond to moment  functionals on  $\dR[x_1,\dots,x_d]$ and the above problem can be rephrased by asking when is each linear functional  $L:\dR[x_1,\dots,x_d]\to \dR$  a  moment functional with signed representing  measure supported by $K$.

The corresponding  general problem  for an  arbitrary commutative finitely generated real unital algebra  is completely  settled in Theorem \ref{main1} by giving  necessary and sufficient conditions.  The case of polynomial algebras is obtained as an application of Theorem \ref{main1} and is stated in Theorem \ref{main2}.

This paper is organized as follows.
In Section \ref{prelim}  we  introduce 
 the necessary  terminology and state our main result (Theorem \ref{main1}).  In Section \ref{polyn} we give the corresponding applications to polynomial algebras  (Theorem \ref{main2}) and develop some corollaries and examples

Sections \ref{prelimth2} and \ref{proofth2} are devoted to the proof of  Theorem \ref{main1}. The crucial technical ingredient is  a result about topologies on unbounded operator algebras which was proved in \cite{sch80}. Note that Sherman's theorem \cite{sherman} (even more, a generalization to enveloping algebras of Lie algebras) has been derived in \cite{sch78} in a similar manner.
\section{Terminology and Main Results}\label{prelim}
 Throughout this paper, $\sA$ denotes a {\bf commutative finitely generated real unital algebra} and $\{a_1,\dots,a_d\}$ is a fixed set of algebra generators of $\sA$. The unit element of $\sA$ is denoted by $1$. Our guiding examples of algebras $\sA$ are the polynomial algebras $\dR[x_1,\dots,x_d]$ for $d\in \dN$.

 A {\it character} of $\sA$ is an algebra homomorphism $\chi:\sA\mapsto \dR $ such that $\chi(1)=1$. Let $\hat{\sA}$ denote the set of characters of $\sA$.

Since  $\{a_1,\dots,a_d\}$ is a  set of algebra generators,   there exists a unique surjective  unital algebra homomorphism $\pi: \dR_d[x]\to {\sA}$ such that $\pi(x_j)=d_j,$  $j=1,\dots,d$. If $\cJ$ denotes the kernel  of $\pi$, then $\cJ$ is an ideal of $\dR_d[x]$ and ${\sA}$ is isomorphic to the quotient algebra\, $\dR_d[x]/ \cJ,$ that is, 
${\sA}\cong \dR_d[x]/ \cJ$.
Each character $\chi$ of ${\sA}$ is  uniquely determined by  the point $x_\chi:=(\chi(a_1),\dots,\chi(a_d))$ of $\dR^d$.  We identify $\chi$ with $x_\chi$ and write $f(x_\chi):=\chi(f)$ for $f\in {\sA}$.  That is, $f(x)$ always denotes the values of the character $x\in \hat{\sA}$ at $f\in \sA$. Under this identification,\, $\hat{{\sA}}$ becomes the real algebraic set 
\begin{eqnarray}\label{Ahatvariety}
\hat{{\sA}}=\cZ(\cJ):=\{x\in \dR^d:p(x)=0~{\rm for} ~p\in \cJ\}.
\end{eqnarray}
Since $\cZ(\cJ)$ is closed in $\dR^d$, $\hat{{\sA}}$ is a locally compact Hausdorff space in the induced topology of $\dR^d$. The elements of ${\sA}$ are real polynomials $p(a_1,\dots,a_d)$ in the generators $a_1,\dots,a_d$  and they are continuous functions on $\hat{\sA}.$ Note that in the  case ${\sA}=\dR_d[x]$ we can take $a_1=x_1,\dots,a_d=x_d$\, and obtain $\hat{{\sA}}= \dR^d$. 

Let $M_+(\hat{\sA})$ denote the set of Radon measures $\mu$ on the locally compact Hausdorff space $\hat{\sA}$ such that all $f\in \sA$ are $\mu$-integrable. Since $1\in \sA$, all measures of $M_+(\hat{\sA})$ are finite. Let $M(\hat{\sA})$ be the set of differences $\mu=\mu_1-\mu_2$ of Radon measures $\mu_1,\mu_2\in M_+(\hat{\sA})$. The elements of $M(\hat{\sA}) $ are  {\it signed} Radon measures. We say that $\mu$ is supported by a subset $K$ of $\hat{\sA}$ if $\supp \mu_1\subseteq K$ and $\supp \mu_2\subseteq K$.
\begin{dfn}
Let $\mu_1,\mu_2 \in M_+(\hat{\sA})$. The linear functional $L$ on $\sA$ defined by
\begin{align*}
L(f)=\int f(x)\, d\mu_1(x)- \int f(x)\, d\mu_2(x), ~~f\in \sA,
\end{align*}
 is called a \emph{ general moment functional} and $\mu:=\mu_1-\mu_2$ is called a {\it representing signed measure} of $L$.
\end{dfn}

We want to study and solve the following problem:\smallskip

 {\it Given a closed subset $K$ of $\hat{\sA}$, when is each linear functional on $\sA$ a general moment functional with representing signed measure supported by $K$? }
\smallskip

This question is  settled  by Theorem \ref{main1} which is  also the main result of this paper. For this  the following linear  subspaces $\cN_n(K)$ of $\sA$ are needed: 
Suppose $K$ is  a  subset of $\hat{\sA}$. For  $n\in \dN_0$ we define
\begin{align*}
\cN_n(K):=\{ a \in &\sA: \textrm{ There exists a number}~~ C_a>0~ \textrm {such that}\\ & |a(x)|\leq C_a (1+a_1(x)^2+ \cdots+ a_d(x)^2)^n~~ \textrm{for all} ~~ x\in K\, \}.
\end{align*}
Further, we  say that $K$ {\it separates the points of $\sA$} if $a(x)=0$ for all $x\in K$ and for some $a\in {\sA}$ implies that  $a=0$.
\begin{thm}\label{main1}   Suppose that  $\{a_1,\dots,a_d\}$, $d\in \dN_0$,  is a  set of generators of the commutative unital real algebra $\sA$ and  $K$ is a closed subset of $\hat{\sA}$. Then the following statements are equivalent:
\begin{itemize}
\item[ (i)] Each linear functional on $\sA$ is a general moment functional with representing measure supported by $K$.
\item[ (ii)]  $K$ separates the points of $\sA$ and   the linear  space $\cN_n(K)$   is finite-dimensional for each $n\in \dN_0$.
\end{itemize}
\end{thm}
The proof of Theorem \ref{main1} will be given in Section \ref{proofth2}. 

The following simple example shows that the separation assumption in Theorem \ref{main1}(ii) cannot be omitted.
\begin{exm}
Let $\sA$ be the quotient algebra of $\dR[x]$ by the ideal generated by $x^2$. That is, $\sA$ is the vector space $\{a+bx; a,b\in \dR\}$ with multiplication rule $(a+bx)(c+dx)=ac+(ad+bc)x$. Then $\sA$ has only one character which is given by $\chi(f)=f(0)$ and each general moment functional is of the form $L(f)=cf(0)$ for some $c\in \dR$. Note that $\hat{\sA}$ does not separate the points of $\sA$ and the linear functional $L_1(f)=f(1)$, $f\in {\sA}$, cannot be represented by some signed Radon measure on  $\hat{\sA}$.
\end{exm}
 \begin{rem}
The considerations and results of this paper extend  easily to each finitely generated commutative unital {\it complex} $*$-algebra $B$, with involution $b\to b^*$, and  linear functionals $L:B\mapsto \dC$. 
It suffices to apply the results for the commutative real algebra ${\sA}:=\{ b\in B: b=b^* \}$. The corresponding representing measures are then complex measures of the form $\mu=\mu_1-\mu_2+\im (\mu_3-\mu_4)$, where $\mu_1,\mu_2,\mu_3,\mu_4$ are (positive) Radon measures.
\end{rem}

\section{Application to polynomial algebras}\label{polyn}

In this section we turn to the special case where $\sA$ is the polynomial algebra $\dR[x_1,\dots,x_d]$ and $a_1=x_1,\dots, a_d=x_d$. As noted above,  $\hat{\sA}=\dR^d$. 
For a  subset $K$ of $\dR^d$ and  $n\in \dN_0$,  we define 
\begin{align*}
\cN_n(K):=\{ p \in &\dR[x_1,\dots,x_d]: \textrm{ There exists a number}~~ \lambda_p >0~ \textrm {such that}\\ & |p(x)|\leq \lambda_p\, (1+x_1^2+ \cdots+ x_d^2)^n~~ \textrm{for all} ~~ x\in K\}.
\end{align*}
Recall that a subset  $K$ of $\dR^d$ is called {\it Zariski dense} if it is not contained in the zero set of a polynomial $p\in \dR[x_1,\dots,x_d], p\neq 0$. Clearly, $K$ is Zariski dense if and only the point evaluations at $K$ separate polynomials of $\dR[x_1,\dots,x_d], p\neq 0$.

The following theorem  restates Theorem \ref{main1} in the present setting.
\begin{thm}\label{main2} Suppose that $K$ is a closed subset of $\dR^d$. Then the following statements are equivalent:
\begin{itemize}
\item[ (i)] Each linear functional on $\dR[x_1,\dots,x_d]$ is a general moment functional with support contained in $K$.
\item[ (ii)]  $K$ is Zariski dense and the linear space $\cN_n(K)$  is finite-dimensional for each $n\in \dN_0$.
\end{itemize}
\end{thm}

Next we give  applications of Theorem \ref{main1}.

The  case $d=1$ is settled completely by the following corollary. It is in fact  Polya's theorem \cite{polya}.
\begin{cor}
Let $K$ be closed subset of 
$\dR$. Then each linear functional on $\dR[x]$ is a general moment functional with support contained in $K$ if and only if $K$ is unbounded.
\end{cor}
\begin{proof}
If $K$ is bounded, then obviously $\dR[x]\subseteq \cN_0(K)$, so condition (ii) in Theorem \ref{main2} is not true. If $K$ is not bounded, then $K$ is infinite, hence  Zariski dense in $\dR$, and it is easily checked that  $\cN_n(K)=\{p\in \dR[x]:\deg p\leq 2n\}$. Thus,  condition (ii) Theorem \ref{main2} is fulfilled.
\end{proof}
In view of Theorem \ref{main2} it might be of interest to characterize the subsets $K$ of $\dR^d$ for which all linear subspaces $\cN_n(K)$ , $n\in \dN_0$, of $\dR[x_1,\dots,x_d]$  are finite-dimensional. Similarly, if $\sA$ is the coordinate algebra of some real algebraic variety, when are all spaces $\cN_n(K)$, as defined in Section \ref{prelim},  finite-dimensional?
\smallskip

Now we suppose that $d\geq 2$.  Then a simple sufficient condition is the following:\smallskip

(*)  {\it For each $j=1,\dots,d$ there exists a Zariski dense subset $M_j$ of $\dR^{d-1}$ such that for each $y=(y_1,\dots,y_{j-1},y_{j+1},\dots,y_d)\in M_j$ there exists a real sequence $(x_n)_{n\in \dN}$ such that\, $\lim_{n \to \infty}|x_n|=+\infty$ and }
\begin{align*}x_n(y):=(y_1, \dots,y_{j-1},x_n, y_{j+1},\dots,y_d)\in K~~\textrm{for} ~~n\in \dN.
\end{align*}

Roughly speaking, condition (*) means that the set $K$ is "unbounded in all directions".
\begin{cor}\label{appl}
Suppose  $d\geq 2$. If $K$ is Zariski dense and condition (*) holds, then each linear functional on $\dR[x_1,\dots,x_d]$ is a general moment functional with signed representing measure whose support is contained in $K$.
\end{cor}
\begin{proof}Suppose that (*) is satisfied. Let $m\in \dN_0$ and $p\in \cN_m(K)$. Fix  $j\in \{1,\dots,n\}$ and $y\in M_j$. We write $p$ as a sum of terms
\begin{align}
p(x)\equiv p(x_1,x_2,\dots,x_d)=\sum_{i=0}^{k_i} p_i(x_1,\dots,x_{j-1},x_{j+1},\dots,x_d)x_j^i
\end{align}
with polynomials $p_j\in \dR[t_1,t_2,\dots,t_{d-1}]$.  

We show that $p_i=0$ for $i>2m$. Assume the contrary. Let $i$ be the largest number such that $p_i\neq 0$.  Since $M_j$ is Zarsiki dense in $\dR^{d-1}$, we can choose $y\in M_j$ such that $p_i(y)\neq 0$. Now we insert the sequence of elements $x_n(y)\in K$ from condition (*) for $x$. Since $p\in \cN_m(K)$, there exists $\lambda_p>0$ such that 
\begin{align}
|p(x_n(y))|\leq \lambda_p (1+y_1^2+\cdots+y_{j-1}^2+x_n^2+y_{j+1}^2+\dots+y_{d}^2)^m
\end{align}
Now we divide both sides by $x_n^{2m}$ and pass to the limit $n \to \infty$. Since $\lim_{n\to \infty}  |x_n|=+\infty$, the left-hand side gives $|p_i(y)|$ and the right-hand side gives zero, a contradiction. 
This proves that $p_i=0$ for $i>2m$. That is, the degree of $p$ with respect to $x_j$ is at most $2m$. Therefore, $\dim \cN_m(K)\leq 2dm$ for all $m\in \dN_0$ and Theorem \ref{main2}  yields the assertion.
\end{proof}

In  particular, condition (*) is fulfilled for $K=\dR^d$ and $K=[0,+\infty)^d$. In this special case Corollary  \ref{appl} gives Sherman's theorem \cite{sherman}. Other simple applications are sets of the form $K=\dR^d\backslash M$ for compact sets $M$.

Note that for $d=2$ the requirement that $M_j$ is Zariski dense in $\dR$ is very simple: It suffices to assume that both sets $M_1$ and $M_2$ are infinite. 
\begin{exm}
Let $d=2$. If $K$ contains a $2$-dimensional affine cone of $\dR^2$, then Corollary \ref{appl} applies, so each linear functional on $\dR[x_1,x_2]$ is a general moment functional with signed representing measure supported by $K$. 

One may also take an infinite grid: If $(x_k)_{k\in \dN}$ and $(y_n)_{n\in \dN}$ are unbounded real sequences, then Corollary  \ref{appl} applies to the set $K:=\{(x_k,y_n): k,n\in \dN\}$.
\end{exm}
\begin{cor}\label{boundedpol} If there is a non-constant polynomial $p\in \dR[x_1,\dots,x_d]$ which is bounded on $K$, then there exist a linear functional  on $\dR[x_1,\dots,x_d]$ that is not  a general moment functional with support contained in $K$.
\end{cor}
\begin{proof}
All powers $p^m$ for $m\in\dN$ are in $\cN_0(K)$, so that  $\cN_0(K)$ has infinite dimension and the assertion follows from Theorem \ref{main2}.
\end{proof} Corollary \ref{boundedpol} applies (for instance) if $K$ is contained in a strip $[a,b]\times \dR$ in $\dR^2$, where $a,b\in \dR, a\leq b.$

\section{Preliminaries to the proof of Theorem \ref{main1}}\label{prelimth2}
The crucial step of the proof of the main implication (ii)$\to$(i) of Theorem 2 is a combination of a result on the uniform topology of an unbounded operator algebra  with a result on normal cones of ordered topological vector spaces.  
In this section we recall the corresponding notions and results. 

We  begin with unbounded operator algebras; proofs and more details can be found in  \cite{sch90} and \cite{sch20}.

Suppose that
$(\cD,\langle \cdot,\cdot \rangle)$ is a complex inner product space. Let $\|\cdot\|$ denote the norm on $\cD$ defined by $\|\varphi\|:=\langle\varphi,\varphi\rangle^{1/2}$, $\varphi\in \cD$. An {\it $O^*$-algebra} on $\cD$ is a subalgebra $\cA$ of the algebra $L(\cD)$ of linear mappings of $\cD$ into itself such that the identity
 map $I_\cD$ is in $\cA$ and for each $a\in \cA$ there exists  $b\in \cA$ satisfying 
\begin{align*}
 \langle a\varphi,\psi\rangle=\langle \varphi,b\psi\rangle\quad
 \text{for}~~ \varphi,\psi\in \cD .
\end{align*}
In this case, $b$ is  uniquely determined by $a$ and denoted by $a^+$. Then $\cA$ is a complex  unital $*$-algebra with involution $a\mapsto a^+$.

Let $\cA$ be an $O^\ast$-algebra on $\cD$. For\, $a\in \cA$\, we define a seminorm $\|\cdot\|_a$  on $\cD$  by\, 
$\|\varphi\|_a:=\|a\varphi\|$, $\varphi \in \cD$. 
The {\it graph topology} $t_\cA$ is the locally convex topology on $\cD$ determined by the family of seminorms $\|\cdot\|_a$, $ a\in \cA$. 
For a bounded subset $M$ of the locally convex space $\cD[t_\cA]$, let $p_M$ be the seminorm on $\cA$ defined by
\begin{align*}
p_M(a):=\sup_{\varphi,\psi\in M} ~|\langle a\varphi,\psi\rangle|,~~ a\in \cA.
\end{align*}
The locally convex topology on $\cA$ defined by the family of such seminorms $p_M$ is called the {\it uniform topology}  and denoted by $\tau_\cD$.

 The uniform topology was introduced by G. Lassner \cite{lassner}; it was extensively studied in the monograph \cite{sch90}. Note that if all operators $a\in \cA$ are bounded, then the graph topology $t_\cA$ is the norm topology of the  norm $\|\cdot\|$ on $\cD$ and the uniform topology $\tau_\cD$ is given by the operator norm. 

 From the polarization identity 
\begin{align*}4\langle a\varphi,\psi\rangle=&\langle a(\varphi+\psi),\varphi+\psi\rangle-\langle a(\varphi-\psi),\varphi-\psi\rangle\\ &+ {\ii}\langle a(\varphi +
{\ii}\psi),\varphi+{\ii} \psi\rangle-{\ii}\langle a(\varphi -{\ii}\psi),\varphi-{\ii} \psi\rangle
\end{align*}
for $a\in \cA$ and $ \varphi,\psi\in \cD$  it follows that the uniform topology $\tau_\cD$ is also generated by the family of seminorms
\begin{align}\label{newsemi}
p'_M(a):=\sup_{\varphi\in M} ~|\langle a\varphi,\varphi\rangle|,~~ a\in \cA.
\end{align}

The crucial technical result for our approach is the following.
\begin{prop}\label{uniform} Suppose  $\cA$ is an $O^*$-algebras on $\cD$ which is countably generated as a $*$-algebra. Then the uniform topology $\tau_\cD$ coincides with the finest locally convex topology $\tau_{st}$ on the vector space $\cA$ if and only if
for each $a\in \cA$ the vector space
\begin{align*}
\cM_a:=\{ b\in \cA: & \textrm{There exists a number}~~ \lambda_b>0~ \textrm{such that}\\ & ~ |\langle b\varphi,\varphi\rangle|\leq \lambda_b \|a\varphi\|^2 ~ \textrm{for all}~ \varphi\in \cD\}
\end{align*}
is finite-dimensional.
\end{prop} 
\begin{proof}
\cite[Theorem 1]{sch80}, see also \cite[Theorem 4.5.4]{sch90}.
\end{proof}

Next we turn to ordered vector spaces, see  \cite[Chapter V]{schaefer} for a detailed treatment. Suppose that $E$ is a real vector space. By a {\it cone} in $E$ we mean a non-empty subset $C$ of $E$ such that $\lambda x\in C$ and $x+y\in C$ for all $x,y\in C$ and $\lambda\in [0,+\infty)$. A linear functional $L:E\to \dR$ is said to be {\it $C$-positive} if $L(x)\geq 0$ for all $x\in C$. 
 
 Let $\tau$ be a locally convex topology on $E$. The vector space of continuous linear functionals $L:E\to \dR$ is denoted by $E[\tau]'$. A cone $C$ of $E$ is called {\it $\tau$-normal} if there exists a generating family $\{p_j; j\in J \}$ of $C$-monotone seminorms for $\tau$, that is, $p_j(x)\leq p_j(x+y)$ for all $x,y\in C$ and $j\in J$ (see e.g. \cite[Chapter V, 3.1]{schaefer}).
 In the terminology of ordered vector spaces the next proposition 
 says that normal cones are weakly normal. This is the second technical ingredient of our proof.
\begin{prop}\label{weaklyn} Suppose that $\tau$ is a locally convex topology on $E$ and $C$ is a $\tau$-normal cone in $E$. Then for each linear functional $L\in E[\tau]'$ there exist $C$-positive linear functionals $L_1,L_2\in E[\tau]'$ such that $L=L_1-L_2$.\end{prop}
\begin{proof} \cite[Corollary 3.1 Chapter V]{schaefer}.
\end{proof}

Now we bring both topics together. Suppose $\cA$ is an $O^*$- algebra on $\cD$. Then
\begin{align*}
\cA_+:=\{a \in \cA: \langle a\varphi,\varphi\rangle \geq 0\quad \textrm{for all}~~ \varphi\in \cD\}
\end{align*}
is a cone in the real vector space 
\begin{align*}
\cA_h:=\{ b\in \cA: b^+=b\}.
\end{align*}
For each bounded subset $M$ of $\cD[t_\cA]$ and $a\in \cA_+$, we have $$p'_M(a)= \sup_{\varphi\in M} ~ \langle a\varphi,\varphi\rangle.$$ This obviously implies that each seminorm $p'_M$ is $\cA_+$-monotone. Since the family of seminorms $p_M'$ defined by  (\ref{newsemi})  generates the topology $\tau_\cD$ as well,  we conclude that the cone $\cA_+$ is normal with respect to the uniform topology $\tau_\cD$ on $\cA_h$.
\section{Proof of Theorem \ref{main1}}\label{proofth2}
In order to apply Proposition \ref{uniform} we have to pass to a complex $*$-algebra. The complexification of the real algebra ${\sA}$ (see e.g. \cite[p. 10]{sch20}) is the direct sum vector space  $A_\dC=A\oplus \im\, A$ with multiplication and involution defined by
\begin{align*}
(a+{\im}\, b)(c+{\im}\, d):= ab-cd+ {\im}\,(bc+ad),\, (a+{\im}\, b)^+:=a-{\im}\, b, ~ a,b,c,d \in A.
\end{align*} 
Then $A_\dC$ is a commutative unital complex $*$-algebra and its hermitean part is $A$, that is, $(A_\dC)_h=\{a\in A_\dC: a^+=a\, \} =A$.
Each character $\chi$ of $\sA$ extends uniquely to a character, denoted again $\chi$,  of ${\sA}_\dC$ by $\chi(a+{\im} b)=\chi(a)+ {\im } \chi(b)$, where $a,b\in {\sA}$. 
Since $\sA$ is finitely generated and the characters of $K$ separate the points of $\sA$, there exists a  sequence $(x_n)_{n\in \dN}$ of characters $x_n \in K$
such that the countable set $K_0:=\{x_n:n\in  \dN\}$  is weakly dense in $K$.  Define
\begin{align*}
\Pos (K) =\{a\in {\sA}:  a(x)\geq 0 ~~ \textrm{for all}~~ x\in K\}.
\end{align*}

Now we consider the Hilbert space $l^2(\dN)$. Let $\cD$ be its dense linear subspace of  ''finite" vectors $(\varphi_1,\dots, \varphi_n,0,0,\dots)$, $n\in \dN$.
We define a mapping $\pi:{\sA}_\dC \to L(\cD)$ by 
\begin{align}
\pi(a)(\varphi_n):=(a(x_n)\varphi),\quad a\in{\sA}_\dC, ~ (\varphi_n)\in \cD.
\end{align} 
Since the $x_n$ are characters of  ${\sA}_\dC$, $\pi$ is an algebra homomorphism and we have $\pi(1)=I_\cD$. From
\begin{align}\label{scalarp}
\langle &\pi(a)(\varphi_n),(\psi_n)\rangle =\sum_n a(x_n)\varphi_n \, \ov{\psi_n}=\sum_n \varphi_n \ov{a^*(x_n)\psi_n}=\langle (\varphi_n),\pi(a^*)(\psi_n)\rangle
\end{align}
we conclude that $\pi$ is a $*$-homomorphism of ${\sA}_\dC$ on some $O^*$-algebra $\cA:=\pi({\sA}_\dC)$. That is, $\pi$ is a $*$-representation of the $*$-algebra ${\sA}_\dC$ with domain $\cD$ (see e.g. \cite[Definition 4.2]{sch20}).
\begin{lem}\label{posi}
$\pi(\Pos(K))=\cA_+\equiv \pi({\sA})_+$.
\end{lem}
\begin{proof}
If $a\in \Pos(K)$, then $a(x_n)\geq 0$ for all $n\in \dN$ and hence $\pi(a)\in \cA_+$ by (\ref{scalarp}). Conversely, let $\pi(a)\in \cA_+$. Then (\ref{scalarp}), applied with $\varphi_n=\psi_n=\delta_{kn}$ for $k,n\in \dN$, yields $a(x_k)\geq 0$. Since the set $K_0$ is weakly dense in $K$, this implies that $a(x)\geq 0$ for all $x\in K$, that is, $a\in \Pos(K)$.
\end{proof}
%Let $b_k:=\pi(1+a_1^2+\cdots+a_d^2)^k$, $k\in \dN_0$.
\begin{lem}\label{finite1}
 Set $b_k:=\pi((1+a_1^2+\cdots+a_d^2)^k)$ for $k\in \dN_0$.
For $c\in \cA$ there exists $k\in \dN_0$ such that $\cM_c\subseteq \cM_{b_k}$.
\end{lem}
\begin{proof}
We write $c=c_1+\ii\, c_2$ with $c_j=c_j^+$ for $j=1,2$. Then there exist $y_j\in {\sA}$ such that $\pi(y_j)=c_j$ for $j=1,2$. Since $\{a_1,\dots,a_d\}$ is a set of generators of the real algebra $\sA$ there exist polynomials $p_j\in \dR[t_1,\dots,t_d]$ such that $c_j=p_j(a_1,\dots,a_d)$. There are numbers $\lambda>0$ and $k\in \dN_0$ such that  $|p_j(t_1,\dots,t_d)|\leq \lambda(1+t_1^2+\cdots+t_d^2)^k$ for all $(t_1,\dots,t_d)\in \dR^d$ and $j=1,2$. 
For $(\varphi_n)\in \cD$ and $j=1,2$, we derive
 \begin{align*}
&\|c_j(\varphi_n)\|^2=\| \pi(y_j)(\varphi_n)\|^2 =\| (y(x_n))(\varphi_n)\|^2=\sum_n |p_j(a_1,\dots,a_d)(x_n)\varphi_n|^2\\&=\sum_n |p_j(a_1(x_n),\dots,a_d(x_n))\varphi_n|^2\leq
\sum_n |\lambda (1+a_1(x_n)^2+\dots+a_d(x_n)^2)^k\varphi_n|^2\\&=\lambda^2\sum_n  |(1+a_1^2+\dots+a_d^2)^k(x_n)\varphi_n|^2=\lambda^2\sum_n |b_k(x_n)\varphi_n|^2 =\lambda^2 \|b_k(\varphi_n)\|^2
\end{align*}
and therefore $\|c(\varphi_n)\|\leq \|c_1(\varphi_n)\|+\|c_2(\varphi_n)\|\leq 2\lambda\, \|b_k(\varphi_n)\|$. This implies that $\cM_c\subseteq \cM_{b_k}$.
\end{proof}

\begin{lem}\label{finite2} Suppose that $\pi$ is injective. Then  $\dim \cM_{b_k}\leq 2\dim \cN_{2k}(K)$.
\end{lem}
\begin{proof}
Suppose that $c=c^+\in \cM_{b_k}$. Then we can find  $y\in A$ such that $c=\pi(y)$. Further, there is a polynomial $p\in \dR[t_1,\dots.,t_d]$ such that  $y=p(a_1,\dots,a_d)$. Let $e_n=(\delta_{ln})_{l\in \dN}$. Since $c\in \cM_{b_k}$, there is $\lambda>0$ such that
\begin{align*}
|y(x_n)|& =|\langle \pi(c) e_n,e_n\rangle| \leq \lambda\, \|\pi(b_k) e_n\|^2=\lambda\, |b_k(x_n)|^2\\ &=\lambda(1+a_1(x_n)^2+\cdots+a_d(x_n)^2)^{2k}
\end{align*}
 and therefore\, $|y(x)|\leq \lambda\, (1+a_1(x)^2+\cdots+a_d(x)^2)^{2k}$\, for all $x\in K$, because the set $K_0$ is weakly dense in $K$. Hence $y\in \cN_{2k}(K)$. 
 
 This shows that $\cM_{b_k}\cap \cA_h \subseteq \pi(\cN_{2k}(K))$. Since $\pi$ is injective, $\dim (\cM_{b_k}\cap \cA_h)\leq \dim \cN_{2k}(K)$. Writing $c\in \cM_{b_k}$ as $c=c_1+{\im} \, c_2$ with $c_1,c_2\in \cA_h$, we have $c_1,c_2\in \cM_{b_k}$. Hence $\dim \cM_{b_k}\leq 2\dim \cN_{2k}(K).$
 \end{proof}
 
After these preparations  we are ready to prove Theorem \ref{main1}.\smallskip

\noindent
{\it Proof of Theorem \ref{main1}}:\\(ii)$\to$(i): 
 By  assumption (ii),  $K$ separates the points of $\sA$ and all spaces $\cN_{n}(K)$, $n\in \dN$,  are finite-dimensional. The first implies that  $\pi$ is injective. Therefore, Lemma \ref{finite1} and \ref{finite2} imply that all spaces $\cM_c$ for $c\in \cA$  are finite-dimensional.  Therefore, by Proposition \ref{uniform}, the uniform topology $\tau_\cD$ on the $O^*$-algebra $\cA$ coincides with the finest locally topology $\tau_{st}$.
 
 Let $L:A\to \dR$ be a linear functional. Then $L'(\pi(a)):=L(a), a\in{\sA},$ defines 
  is a linear functional $L':\cA_h\to \dR$. Recall that the cone $\cA_+$ is $\tau_\cD$-normal. Since $\tau_\cD=\tau_{st}$, $L'$ is $\tau_\cD$-continuous on $\cA_h$. From Proposition \ref{weaklyn}, applied with $C=\cA_+$ and $E=\cA_h$, it follows that $L'$ is the difference $L_1'-L_2'$\, of $\cA_+$-positive functionals. Define $L_j(a):=L_j'(\pi(a))$, $a\in {\sA}$, for $j=1,2$. Then we have $L=L_1-L_2$ by construction. Since 
$\pi(\Pos(K))=\cA_+$ by Lemma \ref{posi}, $L_1$ and $L_2$ are $\Pos(K)$-positive. Hence, by a version of Haviland's theorem \cite{haviland} (in the version stated in  \cite[Theorem 1.14]{sch17}, see also \cite[Theorem 3.2.2]{marshall}) it follows that $L_1$ and $L_2$ can be given by Radon measures on $\hat{A}$ supported by the closed set $K$. Thus, $L=L_1-L_2$ is a general moment functional with representing signed measure supported by $K$. \smallskip

(i)$\to$(ii): First we verify that $K$ separates the points of $\sA$. Assuming the contrary  there exists $a\in {\sA}$, $a\neq 0$, such that $a(x)=0$ for all $x\in K$. Obviously, each  linear functional $L:{\sA}\to \dR$ for which $L(a)\neq 0$ cannot be represented by an integral with respect to some signed measure supported by $K$. This contradicts (i).

Now we prove that all vector spaces $\cN_n(K)$ are finite-dimensional.
Assume to the contrary that there exists $n\in \dN_0$ such that $\cN_n(K)$ is not finite-dimensional. For $a\in \cN_n(K)$ we define
\begin{align*}
q(a):=\inf~ \{ \lambda>0: |a(x)|\leq \lambda\, (1+a_1(x)^2+\cdots+ a_d(x)^2)^n~~\textrm{for all} ~x\in K\, \}.
\end{align*}
 It is straightforward to verify that $q$ is a seminorm on the
  real vector space $\cN_n(K)$. 
  
  We abbreviate $c_n:=(1+a_1^2+\cdots+a_d^2)^n$.
 For $j=1,2$ let $L_j$  be a moment functional on $\sA$ supported  by $K$. It follows at once from the definition of $q(a)$ that $|L_j(a)|\leq q(a)L_j(c_n)$ for all $a\in \cN_n(K)$. Therefore, if $L=L_1-L_2$ is a general moment functional with signed measure supported by $K$, then 
 \begin{align}\label{infiniteqL}
 |L(a)|\leq q(a) (L_1(c_n)+L_2(c_n))\quad \textrm{for}~~ a\in \cN_n(K).
 \end{align}
 If $q(a)=0$ for some $a\in \cN_n(K)$, then $a(x)=0$ for all $x\in K$ and hence $a=0$, because by assumption $K$ separates the points of $\sA$. That is, $q(a)\neq 0$ if $a\in \cN_n(K)$ and $a\neq 0$. We choose a Hamel basis $\{a_j: j\in J\}$ of the infinite-dimensional (!) vector space $\cN_n(K)$ and define a linear functional $L_n$ on $\cN_n(K)$ such that $\sup_{j\in J} |L_n(a_j)|q(a_j)^{-1}=+\infty$. 
 Then  (\ref{infiniteqL}) cannot hold  for $L_n.$ Each extension of $L_n$  to a linear functional $L$ on $\sA$ is not  a general moment functional with representing signed measure supported by $K$. This  contradicts (i).
 
 This completes the proof of Theorem 2.
$\hfill \Box$

\section*{Acknowledgement}

This research was carried out during my stay at the Zukunftskolleg of the University of Konstanz. I would like to thank Dr. Ph. di Dio and Prof. C. Scheiderer for their kind hospitality.
\bibliographystyle{amsalpha}

\begin{thebibliography}{A}

\bibitem[B39]{boas} 
Boas, R. P., The Stieltjes moment problem for functions of bounded variation, Bull. Amer. Math. Soc. {\bf 45}(1939), 399--404.
\bibitem[D89]{duran} 
Duran, A., The Stieltjes moments problem for rapidly decreasing functions,  Proc. Amer. Math. Soc. {\bf 107}(1989), 731--741.
\bibitem[H36]{haviland}
Haviland, E. K., On the moment problem for distribution functions in more than one dimension II, {\bf 58}(1936), 164--168.

\bibitem[L73]{lassner}
Lassner, G.,  Topological algebras of operators, Reports Math. Phys. {\bf 3}(1972),  279--293.


\bibitem[P38]{polya}
Polya, G., Sur l'indetermination d'un probleme voisin du probleme des momentes, C. R. Math. Acad. Sci. Paris {\bf 207}(1938), 708--711.



\bibitem[M08]{marshall} Marshall, M.,  \textit{ Positive Polynomials and Sums of Squares}, Math. Surveys and Monographs, Amer. Math. Soc., Providence, R.I., 2008.


\bibitem[Sch99]{schaefer} 
Sch\"afer, H.H.: \textit{Topological Vector Spaces}, Graduate Texts in Math. {\bf 3}, Sec. Edition, Springer-Verlag, New York, 1999.


\bibitem[S78]{sch78}  Schm\"udgen, K., Uniform topologies on enveloping algebras, J. Funct. Anal. {\bf 39}(1978), 57--66.
\bibitem[S80]{sch80}  Schm\"udgen, K., Two theorems about topologies on countably generated Op*-algebras, Acta Math. Acad. Sci. Hungar.  {\bf{35}}(1980), 139--150.
\bibitem[S90]{sch90}
Schm\"udgen, K.:  \textit{Unbounded Operator Algebras and Representation Theory}, Operator Theory: Advances and Applications {\bf 37}, Birkh\"auser-Verlag,  Basel, 1990.

\bibitem[S17]{sch17} Schm\"udgen, K.,  \textit{ The Moment Problem}, Graduate Texts in Math., Springer, Cham, 2017.
\bibitem[S20]{sch20}
Schm\"udgen, K.,  \textit{An Invitation to Unbounded Representations of $*$-Algebras on Hilbert Spaces}, Graduate Texts in Math. {\bf 285}, Springer, Cham, 2020.


\bibitem[Sh64]{sherman}
Sherman, Th., A moment problem on $R^N$, Rend. Circ. Matem. Palermo {\bf XIII}(1964).

\bibitem[ST43]{shohat} 
Shohat. J.A. and Tamarkin, J.D., The Problem of Moments, Amer. Math. Soc. Providence, R.I. , 1943.




\end{thebibliography}

\end{document}